\documentclass{amsart}  

\usepackage{amsfonts,amssymb,amsmath,amsthm}

\usepackage{geometry}

\title[Purely infinite corona algebras and extensions]{A note on purely infinite corona algebras and extensions}

\author{P. W. Ng}

\address{Department of Mathematics\\  University of Louisiana at Lafayette\\
217 Maxim Doucet Hall\\   1401 Johnston St.\\
Lafayette, Louisiana\\
70503\\  USA}

\email{png@louisiana.edu}

\author{Cangyuan Wang}

\address{School of Mathematical Sciences\\
	Ocean University of China\\
	238 Songling Road, Laoshan District\\
	Qingdao, Shandong\\
266100\\ 
P. R. China}

\email{cyw@ouc.edu.cn}

\newtheorem{thm}{Theorem}[section]  
\newtheorem{lem}[thm]{Lemma}
\newtheorem{prop}[thm]{Proposition}  

\newtheorem{df}[thm]{Definition}

\theoremstyle{definition}
\newtheorem{para}[thm]{}

\newcommand{\A}{\mathcal{A}}
\newcommand{\B}{\mathcal{B}}
\newcommand{\C}{\mathcal{C}}
\newcommand{\D}{\mathcal{D}}
\newcommand{\E}{\mathcal{E}}

\newcommand{\I}{\mathcal{I}}
\newcommand{\J}{\mathcal{J}}
\newcommand{\K}{\mathcal{K}} 
\newcommand{\Z}{\mathcal{Z}} 

\newcommand{\M}{\mathbb{M}}

\newcommand{\Mul}{\mathcal{M}}

\newcommand{\oo}{\mathcal{O}}

\newcommand{\ee}{\mathbf{Ext}}

\numberwithin{equation}{thm}

\begin{document}

\maketitle

\begin{abstract}
Let $\A$ be a separable nuclear C*-algebra, and $\B$ be a nonunital separable simple
$\Z$-stable C*-algebra.
Continuing the work from \cite{GabeLinNg}, we classify all essential extensions,
with large complement, of the form 
$$0 \rightarrow \B \rightarrow \E \rightarrow \A \rightarrow 0,$$
for the following cases:
i. $\C(\B)$ is properly infinite, and the extension is full.
ii. $\C(\B)$ is purely infinite (though not necessarily simple). 
We also have some more general results. 
\end{abstract}

\section{Introduction}

In their groundbreaking work, Brown, Douglas and Fillmore (BDF) classified all essentially normal operators, on a separable
infinite dimensional Hilbert space, using K theory  (\cite{BDF1}).  Their classification amounted to the classification
of all essential extensions of the form
$$0 \rightarrow \K \rightarrow \E \rightarrow C(X) \rightarrow 0,$$
where $\K$ is the C*-algebra of compact operators on a separable Hilbert space $l_2$, and $X$ is a compact
subset of the plane.   This was a starting point for many developments in operator theory and noncommutative topology, 
with the introduction
of many powerful tools like KK theory. (Some examples of the history and
 broader context can be found in 
\cite{BDF1}, \cite{BDF2}, \cite{DouglasWeylLectures}, \cite{VoiculescuWvN}, \cite{ArvesonExt}, \cite{KasparovAbsorbing}, \cite{KasparovKK},
\cite{RosenbergSchochet},
\cite{LinStableAUE}, \cite{HadwinMaShen}, etc.)

The BDF--Kasparov theory classifies absorbing extensions.  (Under a nuclearity hypothesis, Kasparov's $KK^1$ computes
the unitary equivalence classes of absorbing extensions (\cite{KasparovKK});  and the original BDF theorem (\cite{BDF1})
can be viewed as a
piece of the Universal Coefficient Theorem (\cite{RosenbergSchochet}) for $KK^1$).    In fact, absorbing 
extensions are at the heart of the classical extension theory (e.g., \cite{BDF1}, \cite{ArvesonExt}, \cite{VoiculescuWvN},
\cite{KasparovAbsorbing}, \cite{KasparovKK}, \cite{ThomsenAbsorbing}).
The key problem is that the class of absorbing extensions can be a very thin class.  Among other things, the BDF--Kasparov
theory only classifies extensions of stable C*-algebras;  and also,  every absorbing extension
$\phi : \A \rightarrow \Mul(\B \otimes \K)/(\B \otimes \K)$ is \emph{full} -- i.e., for all $a \in \A - \{ 0 \}$, 
$Ideal(\phi(a)) = \Mul(\B \otimes \K)/(\B \otimes \K)$.   Since for a general C*-algebra $\B$, 
$\Mul(\B \otimes \K)/(\B \otimes \K)$ can have many (even infinitely many) ideals,  $KK^1$ misses a lot of 
essential extensions.

Perhaps one reason for the smoothness of the original BDF theory (\cite{BDF1}) is that
 their corona algebra, which is the
Calkin algebra $\mathbb{B}(l_2)/\K$, is simple purely infinite.  Among other things, the BDF--Voiculescu theorem, 
which roughly speaking says that ``all essential extensions of $\K$ are absorbing", would not be true if 
 $\mathbb{B}(l_2)/\K$ were not simple.   Thus, a reasonable approach in generalizing  BDF theory to more general contexts would be to first single out
 classes of corona algebras with ``nice properties", and then study extension theory in 
this new 
``nice" setting.  This idea was already present in multiple early works (e.g., \cite{ElliottHandelman},
\cite{LinExtII}), but the early definitive results were due to Lin (e.g., \cite{LinExtII}).
Lin characterized the class of simple purely infinite corona algebras (\cite{LinContScale}):  Let $\B$ be a $\sigma$-unital,
simple, nonunital C*-algebra.  Then $\Mul(\B)/\B$ is simple if and only if $\Mul(\B)/\B$ is simple purely infinite
if and only if either $\B \cong \K$ or $\B$ has continuous scale.  (Notice that in this characterization, $\B$ need not be
stable.)  We also note that for quite general extensions $\phi : \A \rightarrow \C(\D)$ (for quite general $\A$ and $\D$),
$\phi$ often has a ``piece" which lives inside a minimal ideal of $\C(\D)$, and this ``piece" will often be an extension of 
a  continuous scale C*-algebra (e.g., see \cite{LinExtIII}).

Generalizing previous works, \cite{GabeLinNg} gave the following classification result:  Let $\A$ be a separable 
nuclear C*-algebra and let $\B$ be a nonunital $\sigma$-unital simple continuous scale C*-algebra.  Then the 
map $Ext(\A, \B) \rightarrow KK(\A, \Mul(\B)/\B) : [\phi]_{Ext} \mapsto [\phi]_{KK}$ is a group isomorphism.  Here,
$Ext(\A, \B)$ consists of the unitary equivalence classes of nonunital essential extensions $\A \rightarrow \Mul(\B)/\B$.

This paper generalizes the result in \cite{GabeLinNg} in two directions:   
i.  Replace $\B$ by a simple nonunital $\Z$-stable C*-algebra for with $\Mul(\B)/\B$ is properly infinite, and restrict
to full nonunital extensions.  ii. Replace $\B$  by a simple nonunital $\Z$-stable C*-algebra for with $\Mul(\B)/\B$ is purely
infinite (though not necessarily simple -- so nonfull extensions are also classified).
There are also some more general results.
This paper uses many techniques and ideas from classical extension theory, the paper 
\cite{GabeLinNg}, as well
as the Elliott program for classifying nuclear C*-algebras using 
K theory invariants. 

The first author thanks J. Gabe for pointing out an early version of Theorem  
\ref{thm:ClassifyFullCase}.  

\section{Some preliminaries}
\label{sect:SomePreliminaries}

\begin{para}
Some references, for basic extension theory, are \cite{BlackadarKTh}, \cite{JensenThomsen} and
\cite{WeggeOlsen}. 
For a nonunital C*-algebra $\B$, we let $\Mul(\B)$ and $\C(\B) =_{df}
\Mul(\B)/\B$ denote the multiplier algebra and corona algebra of $\B$
respectively. Recall that for any C*-algebra $\A$, there is, up
to strong isomorphism, a 
one-to-one correspondence between extensions of 
$\B$ by $\A$ (i.e., extensions   
$0 \rightarrow \B \rightarrow \E \rightarrow \A \rightarrow 0$) and
*-homomorphisms $\phi : A \rightarrow \C(\B)$ (the 
\emph{Busby invariant} of the extension); e.g., see \cite{BlackadarKTh} 15.2, 15.3, 15.4.
Recall that an extension is essential 
if and only if the corresponding Busby invariant is
injective.   
Since we are only interested in properties of an extension that are invariant
under strong isomorphism,  
following standard convention, 
we identify a *-homomorphism $\A \rightarrow \C(\B)$
with the corresponding extension.   
\end{para}
\begin{para} \label{para:DfUE} 
Let $\phi, \psi : \A \rightarrow \C(\B)$ be extensions.
Recall that $\phi$ and $\psi$ are said to be
\emph{weakly equivalent} 
if there exists a partial isometry $v \in \C(\B)$ 
with $v^*v \phi(a) = \phi(a)$ for all $a \in \A$
such that
$v \phi(\cdot) v^* = \psi(\cdot).$
If we can additionally find a unitary 
$U \in \Mul(\B)$ for which $v = \pi(U)$, then we say that
$\phi$ and $\psi$ are \emph{unitarily equivalent}.
Moreover, we let $[\phi]$ or $[\phi]_{Ext}$ denote the 
unitary equivalence class of $\phi$.  
Next, $\phi$ and $\psi$ are said to be \emph{asymptotically 
	unitarily equivalent}  if there exists a norm-continuous path
$\{ u_t \}_{t \in [1, \infty)}$, of unitaries in $\C(\B)$ such
that for all $a \in \A$,
$u_t \phi(a) u_t^* \rightarrow \psi(a).$ 
Also, $\phi$ is said to have \emph{large complement} if 
there exists a projection $p \in \phi(\A)^{\perp}$ for which
$1_{\C(\B)} \preceq p$.  
\end{para}
\begin{para} \label{para:BDFSum} 
Suppose, in addition, that $\C(\B)$ is properly infinite.
The \emph{Brown--Douglas--Fillmore sum (BDF sum)} of two extensions
$\phi, \psi : \A \rightarrow \C(\B)$ is given by 
$(\phi \oplus \psi)(\cdot) =_{df} 
S^* \phi(\cdot) S + T^* \psi(\cdot) T,$ where $S, T \in \C(\B)$
are isometries for which $S S^* + T T^* \leq 1_{\C(\B)}$.   
The BDF sum is well-defined up to weak equivalence.  
When $\C(\B)$ is also $K_1$-injective and 
both $\phi$ and $\psi$ have large complement, then the BDF sum
is well-defined up to unitary equivalence (see \cite{GabeLinNg}       
Proposition 3.7 or \cite{ChandNgSutradhar} Lemma 2.2 for a more
elementary proof).    
\end{para}
\section{Classifying extensions:  The full case}

\begin{para}\label{para:CuntzSubequiv}
We fix a notation: If $\D$ is a C*-algebra and $a, b \in \D_+$, 
$a \preceq b$ if there exists a sequence $\{ x_n \}$ in $\D$ such
that $x_n b x_n^* \rightarrow a$ in norm.  
\end{para}
\begin{para} \label{para:fpi}
Let $\A, \B$ be C*-algebras with $\B$ nonunital.  Let 
$\phi : \A \rightarrow \C(\B)$ be an extension.
$\phi$ is said to be \emph{full} if for
all $a \in \A - \{ 0 \}$, $\phi(a)$ is full in $\C(\B)$ (i.e., 
$Ideal(\phi(a)) = \C(\B)$).  
$\phi$ is said to be \emph{purely infinite} if for all
$a \in \A_+ - \{ 0 \}$, $\phi(a) \oplus \phi(a) \preceq \phi(a)$
in $\M_2(\C(\B))$.  
It is an exercise to show that $\phi$ is both full and purely infinite
if and only if for all $a \in \A_+ - \{ 0 \}$, there exists an
$x \in \C(\B)$ for which 
\begin{equation} \label{equ:fpiCharacterization}
	x \phi(a) x^* = 1_{\C(\B)}. \end{equation}
\end{para}
\begin{df} \label{df:Extfpi} 
	Let $\A$ and $\B$ be C*-algebras with $\B$ being nonunital.

	We let $\ee_{fpi}(\A, \B)$ denote the set of all unitary 
	equivalence classes of full, purely infinite extensions
	$\phi : \A \rightarrow \C(\B)$ with large complement.
\end{df} 

\begin{para} \label{para:ExtNonempty}
Suppose, in addition, that we assume that $\C(\B)$ is properly infinite.
Then $\C(\B)$ contains a unital copy of $O_{\infty}$, and hence, a full copy of 
$O_2$.  If also, $\A$ is separable and nuclear, then by \cite{KirchbergPhillips},
$\ee_{fpi}(\A, \B)$ is nonempty.
\end{para}

\begin{para} \label{para:ExtfpiBDFSUM}
Again, suppose, in addition, that
we assume that $\C(\B)$ is properly infinite. 
From paragraph \ref{para:BDFSum}, the BDF sum induces a 
well-defined addition on $\ee_{fpi}(\A, \B)$ which is 
given by $[\phi] + [\psi]=_{df} [\phi \oplus \psi]$, which
we also call the \emph{BDF sum}.   
Moreover, with even more (relatively mild and reasonable)
hypotheses, we get group structure.
\end{para}
\begin{thm} \label{thm:ExtfpiIsGroup}
	Let $\A$ be a separable nuclear C*-algebra, and let $\B$
	be a nonunital $\sigma$-unital simple C*-algebra for which
	$\C(\B)$ is properly infinite and $K_1$-injective.  
	
	Then $\ee_{fpi}(\A, \B)$, with the addition operation given
	in the previous paragraph (BDF sum), is an abelian group.
\end{thm}

\begin{proof}
	This is \cite{ChandNgSutradhar} Theorem 2.14.   
\end{proof}

We give a (different) proof of
\cite{GabeLinNg} Proposition 4.3. 
The following proof is actually due to J. Gabe and H. Lin. 
Like the argument of \cite{GabeLinNg} Proposition 4.3, the present
argument is
a quick argument where the
result 
follows immediately from Phillips and Weaver's interesting paper \cite{PhillipsWeaver} Proposition 1.4.  
Any lengthiness is due to the notation.   (See also \cite{ManuilovThomsen}.)

\begin{lem} \label{lem:PhillipsWeaver1} (J. Gabe and H. Lin;  cf. \cite{GabeLinNg} Lemma 4.3,
	\cite{ManuilovThomsen}, and \cite{PhillipsWeaver} Proposition 1.4.)   
	Let $\D$ be a separable C*-algebra and $\E$ a nonunital
	$\sigma$-unital C*-algebra.  Two extensions $\phi, \psi
	: \D \rightarrow \C(\E)$ are asymptotically unitarily equivalent
	if and only if they are weakly unitarily equivalent. 
\end{lem}

\begin{proof}   
	Since $\phi$ and $\psi$ are asymptotically unitarily equivalent,
	$ker(\phi) = ker(\psi)$, and so, for simplicity, let us assume
	that $\phi$ and $\psi$ are both injective. 
	
	Let $C_b([1,\infty), \C(\E))$ be the C*-algebra of bounded
	continuous functions from $[1,\infty)$ to $\C(\E)$, and let
	$C_0([1, \infty), \C(\E))$ be the ideal consisting of such functions
	which vanish at infinity.  
	Consider the quotient C*-algebra     
	$\mathcal{Q} =_{df} C_b([1, \infty), \C(\E))/ C_0([1, \infty), \C(\E)),$ and
	let $\pi_Q : C_b([1,\infty), \C(\E)) 
	\rightarrow \mathcal{Q}.$    
	be the quotient map.
	
	Since $\phi$ and $\psi$ are asymptotically unitarily equivalent, 
	let $\{ u_t \}_{t \in [1, \infty)}$ be a norm continuous path of unitaries
	in $\C(\E)$ which realizes this (i.e., $u_t \phi(\cdot) u_t^* \rightarrow 
	\psi(\cdot)$ in the pointwise-norm topology). 
	Let $U \in C_b([1, \infty), \C(\E))$ be the unitary given by $U(t) =_{df} u_t$
	for all $t \in [1, \infty)$.   
	
	Let $\iota_C : \C(\E) \rightarrow \mathcal{Q}$ be the injective *-homomorphism 
	given by taking the
	constant maps in the obvious way.  I.e., 
	for all $x \in \C(\E)$,
	$\iota_C(x) =_{df} \pi_Q(\widetilde{x})$, 
	where $\widetilde{x} \in C_b([1,\infty), \C(\E))$ is the constant function
	$\widetilde{x}(t) =_{df} x$ for all $t \in [1, \infty)$.
	
	Let $\B \subseteq \mathcal{Q}$ be the separable C*-subalgebra  
	given by $\B =_{df} C^*(\iota_C \circ \phi(\D), \iota_C \circ \psi(\D), \pi_Q(U)),$
	and let $\A \subseteq \B$ be the separable C*-subalgebra given by
	$\A =_{df} C^*(\iota_C \circ \phi(\D), \iota_C \circ \psi(\D)) = 
	\iota_C(\A_0)$ 
	where $\A_0 \subseteq \C(\E)$ is the C*-subalgebra given 
	by 
	$\A_0 =_{df} C^*(\phi(\D), \psi(\D)).$
	In fact, $\iota_C$ restricts to a *-isomorphism from $\A_0$ onto $\A$. 
	Hence, $\iota_C^{-1}$ is well-defined on $\A$.
	
	Now for the inclusion map $\iota_B : \B \hookrightarrow \mathcal{Q}$,
	the restriction $\iota_B |_{\A}$
	factors through the map $\rho : \A \rightarrow \C(\E)$ 
	where    
	$\rho =_{df} \iota_{\A_0} \circ \iota_C^{-1} : 
	\A \rightarrow \A_0  \rightarrow \C(\E),$
	and where $\iota_{\A_0} : \A_0 \hookrightarrow \C(\E)$
	is the inclusion map.    
	(In fact, $\iota_B |_{\A} = \iota_C \circ \rho$.)
	
	Hence, by \cite{PhillipsWeaver} Proposition 1.4, there exists a
	*-homomorphism $\Psi : \B \rightarrow \C(\E)$ for which
	\begin{equation} \label{equ:Sept32024} \Psi |_{\A} = \rho. \end{equation}
	
	Note that in $\B$, we have that for all $d \in \D$,
	$\pi_C(U) (\iota_C\circ \phi(d)) \pi_C(U)^* = \psi(d).$  
	Hence, by (\ref{equ:Sept32024}) and the definitions of $\rho$
	and $\Psi$, in $\C(\E)$,   
	there exists a unitary $w \in \C(\E)$ for which
	$$w \phi(d) w^* = \psi(d) \makebox{  for all  } d \in \D.$$
	\end{proof}

\begin{prop} \label{prop:fpiCommutant}
	Let $\B$ be a nonunital simple $\sigma$-unital C*-algebra such that
	$C(\B)$ is properly infinite.  Let $\A$ be a unital
	separable nuclear C*-algebra
	and
	$\phi: \A \rightarrow C(\B)$ be a unital *-monomomorphism
	which is full and properly infinite; i.e.,
	for all $a \in \A_+ - \{ 0 \}$, there exists an $x \in C(\B)$ for which
	\begin{equation} x \phi(a) x^* = 1_{C(\B)}.
		\label{equ:Feb1020216AM}
	\end{equation}  
	\noindent(See \ref{para:fpi}.)

	Then $\phi(\A)'$ is properly infinite. 
	\label{prop:OinftyStable}
\end{prop}

\begin{proof}
	
	Consider the direct sum (or ``diagonal sum")
 $\rho =_{df} \phi \oplus 0$  
	which maps $\A$ into $\M_2 \otimes C(\B)$.
	Since $\C(\B)$ is properly infinite, the map
	$\rho$ 
	still satisfies (\ref{equ:Feb1020216AM}), as a map into
	$\M_2 \otimes \C(\B)$ (though $\rho$ is not unital).
	Hence, by \cite{ChandNgSutradhar} Theorem 2.9, 
	there exists a $v \in \M_2 \otimes C(\B)$ such that 
	\begin{equation}
		v \rho(a) v^* = \phi(a) \oplus \phi(a)
\makebox{ in } \M_2 \otimes \C(\B)		
\label{equ:June29202010AM}
	\end{equation} 
	for all $a \in \A$, where $\oplus$ is the ``diagonal sum".

	Clearly, we may assume that 
	$v^*v \leq 1 \oplus 0 
		\makebox{ and }   
		vv^* = 1 \oplus 1.$   
	
	Since $\rho$ and $\phi \oplus \phi$ are both *-homomorphisms,
	it follows from (\ref{equ:June29202010AM}) that
	%\begin{eqnarray*} 
	\[
	v \rho(a) v^* v \rho(a) v^*
	=  (\phi(a) \oplus \phi(a))^2
	=   \phi(a^2) \oplus \phi(a^2)\\
	=  v \rho(a^2) v^*\\
	=  v \rho(a) \rho(a) v^* 
	\]
	%\end{eqnarray*}
	for all $a \in \A$.  In the above, all $\oplus$ are ``diagonal
sums".

	Hence, 
	$v^* v \rho(a) (1 - v^*v) \rho(a) v^*v = 0$ 
	for all $a \in \A$.
	Hence, the projection $v^*v$ commutes with every element of $\rho(\A)$.

	Therefore, in $\M_2 \otimes \C(\B)$,
	\begin{eqnarray*}
		(\phi(a) \oplus \phi(a)) v 
		& = & v (\phi(a) \oplus 0) v^* v 
		\makebox{  (from (\ref{equ:June29202010AM}) and the definition of $\rho$)}\\ 
		& = & v (\phi(a) \oplus 0)  \makebox{  (since }\rho(a) \makebox{  commutes
			with  } v^*v \makebox{)}\\
		& = & v(\phi(a) \oplus \phi(a))  \makebox{   (since  } v^*v \leq 1 \oplus 0
		\makebox{)}
	\end{eqnarray*}
	for all $a \in A$.
In the above, all $\oplus$ are ``diagonal sums".
	
	Therefore, $v \in \M_2(\phi(A)').$
	
	Therefore, $v$ witnesses that 
	$1 \oplus 1 \preceq 1 \oplus 0$
	in $\M_2(\phi(A)')$.
	
\end{proof}

\begin{para} \label{para:ForcedUnitization} For a C*-algebra $\D$, we let 
$\D^+$ denote the \emph{forced unitization} of $\D$. I.e., 
if $\D$ is nonunital, then $\D^+$ is the minimal unitization of $\D$; and
if $\D$ is unital, then $\D^+ =_{df} \D \oplus \mathbb{C}$.
\end{para}

We are now ready to
classify full, purely infinite extensions; i.e., 
we compute $\ee_{fpi}(\A, \B)$.  The argument is very similar
to \cite{GabeLinNg} Theorem 4.4, and we thank J. Gabe for pointing out
part of the argument to us.

\begin{thm} \label{thm:ClassifyFullCase}
	Let $\A$ be a separable, nuclear C*-algebra, and let $\B$ be 
	a nonunital, $\sigma$-unital, simple C*-algebra for which 
	$\C(\B)$ is properly infinite and $K_1$-injective. 
	
	Then the canonical homomorphism 
	$$\Gamma: \ee_{fpi}(\A, \B) \rightarrow KK(\A, \C(\B)) : [\phi]_{Ext} 
	\mapsto [\phi]_{KK}$$
	is a group isomorphism.
\end{thm}

\begin{proof}
	
	It is clear, from basic KK theory,
	that $\Gamma$ is a well-defined group homomorphism.  
	%(Note that in an abelian group, the zero element is the unique element
	%satisfying the identity $e + e = e$.)

	Firstly, we check that $\Gamma$ is surjective. 
	Let
	$x$ in $KK(\A,\C(\B))$ be arbitrary.
	By Theorem A of \cite{GabeOinfty}, there exists a 
	*-homomorphism $\phi_1 : \A\rightarrow \C(\B)$ such that $[\phi_1]_{KK}=x$. 
	By Theorem \ref{thm:ExtfpiIsGroup}, let $\phi_0 : \A \rightarrow \C(\B)$
	be a full, purely infinite extension, with large complement,
	such that $[\phi_0]_{Ext} = 0$ in 
	$\ee_{fpi}(\A, \B)$.  Then $\phi_0 \oplus \phi_1$ is a full purely infinite 
	extension with large complement, where $\oplus$ is the BDF sum (see
	\ref{para:BDFSum}; note that there exists $x \in \C(\B)$ such that 
	$x (\phi_0 \oplus \phi_1)x^* = 1_{\C(\B)}$, since a similar equation
	holds for $\phi_0$ -- see \ref{para:fpi} (\ref{equ:fpiCharacterization})).   
	So $[\phi_0 \oplus \phi_1]_{Ext} \in \ee_{fpi}(\A, \B)$, and
	$$\Gamma([\phi_0 \oplus \phi_1]_{Ext}) = \Gamma([\phi_0]_{Ext} 
	+ [\phi_1]_{Ext})
	= \Gamma([\phi_1]_{Ext}) = x.$$

	Secondly, we show that $\Gamma$ is injective.  
	Assume that $\phi,\psi: \A\rightarrow \C(\B)$ are full purely 
	infinite extensions with large complement such that  
	$\Gamma([\phi]_{Ext})=\Gamma([\psi]_{Ext})$; i.e., 
	$[\phi]_{KK}=[\psi]_{KK}$.

	Let $\A^+$ be the forced unitization of $\A$ (see
	\ref{para:ForcedUnitization}).
	Let $\phi^+, \psi^+ : \A^+ \rightarrow \C(\B)$ be the *-homomorphisms 
	which are the unique unital extensions of $\phi$ and $\psi$ respectively. 
	I.e.,  $\phi^+(1_{\A^+})=1_{\C(\B)}$, $\phi^+|_{\A}=\phi$,
	$\psi^+(1_{\A^+})=1_{\C(\B)}$, and $\psi^+|_{\A}=\psi$.    
	
	Next, let us prove that $\phi^+$ and $\psi^+$ both are full and purely
	infinite.  It suffices to prove that $\phi^+$ is full and purely infinite
	(since the proof for $\psi^+$ is exactly the same). 
	Let $a \in \A^+$ be a nonzero positive element.  If 
	$a \in \A$,  then, since $\phi^+ |_{A} = \phi$ and $\phi$ is full and
	purely infinite, we can find $x \in \C(\B)$ with $x \phi(a) x^* = 1_{\C(\B)}$  (see \ref{para:fpi} (\ref{equ:fpiCharacterization})); and so, $x \phi^+(a) x^*  = 1_{\C(\B)}$.
	Now suppose that $a = \alpha 1 + a_0$ where $\alpha > 0$ and $a_0 \in \A$.
	Then $\phi(a) = \alpha 1_{\C(\B)} + \phi(a_0)$. Since $\phi$ has   
	large complement, we can find a projection $p \in \phi(\A)^{\perp}$ such that
	$p \sim 1_{\C(\B)}$ (as a consequence, $p \perp \phi(a_0)$). Hence, let 
	$v \in \C(\B)$ be a partial isometry where $v^* v = p$ and $v v^* 
	= 1_{\C(\B)}$.  Hence, if we
	take $x =_{df} \frac{1}{\sqrt{\alpha}} v$, then
	$x \phi^+(a) x^* = 1_{\C(\B)}$.  Since $a \in (\A^+)_+ - \{ 0 \}$ was arbitrary,
	$\phi^+$ is full and purely infinite   (see \ref{para:fpi} (\ref{equ:fpiCharacterization})).    
	By the same argument, $\psi^+$ is full and purely infinite.

	By Proposition \ref{prop:fpiCommutant} and since properly infinite
	C*-algebras have a unital copy of $O_{\infty}$,  
	$\phi^+$ and $\psi^+$ both satisfy the property of being 
	\emph{strongly $O_{\infty}$-stable} in the sense of \cite{GabeOinfty} (1.1)
	(see also \cite{GabeOinfty} Definition 4.1 and Remark 4.2).
	Also, it is not hard to check (e.g., by looking at the six term exact
	sequence in KK for $0 \rightarrow 
	\A \rightarrow \A^+ \rightarrow \mathbb{C} \rightarrow 0$), that
	since $[\phi]_{KK} = [\psi]_{KK}$ in $KK(\A, \C(\B))$,
	$[\phi^+]_{KK} = [\psi^+]_{KK}$ in $KK(\A^+, \C(\B))$. 
	Hence, by \cite{GabeOinfty} Theorem B,  $\phi^+$ and $\psi^+$  
	are asymptotically unitarily equivalent. Hence, 
	$\phi$ and $\psi$ are asymptotically unitarily equivalent.  Hence, by 
	Lemma \ref{lem:PhillipsWeaver1} (see also \cite{GabeLinNg} Proposition 4.3),
	$\phi$ and $\psi$ are weakly unitarily equivalent.  Hence, by \cite{GabeLinNg}
	Proposition 3.7 (see also \cite{ChandNgSutradhar} Lemma 2.2 for a more 
	elementary proof),  
	$\phi$ and $\psi$ are unitarily equivalent; i.e., 
	$[\phi]_{Ext} = [\psi]_{Ext}$ in $\ee_{fpi}(\A, \B)$.\\  
\end{proof}

\begin{df}  \label{df:StrongCFP} 
	Let $\B$ be a $\sigma$-unital, 
	nonunital C*-algebra for which $\C(\B)$ is properly infinite.
	
	Then $\B$ is said to have the \emph{strong corona factorization property
		(strong CFP)} if for all $A \in \C(\B)_+$, if $A$ is full in  
	$\C(\B)$ (i.e., $Ideal(A) = \C(\B)$), then there exists an $X \in \C(\B)$
	such that 
	$$X A X^* = 1_{\C(\B)}.$$ 
\end{df}

\begin{para}
	The strong corona factorization property (strong CFP) 
	differs from the usual corona factorization property (CFP) 
	(\cite{KucerovskyNgCFPAUE};  \cite{NgCFPSurvey})   in that in 
	the original CFP, the canonical ideal $\B$ is assumed to be stable, and
	the positive element $A$ is assumed to have the property that every nonzero 
	element of $C^*(A)$ is full in $\C(\B)$.
	
	Recall that for a separable stable C*-algebra $\B$, $\B$ has the CFP
if and only if every norm-full extension of $\B$ is nuclearly absorbing (\cite{KucerovskyNgCFPAUE},
\cite{NgCFPSurvey});  and thus, 
the CFP characterizes a class of C*-algebras for which there is a concrete way to 
single out the extensions which can be classified by $KK^1$, as in the
BDF--Kasparov theory.	
We generalize this to our setting using the strong CFP (see Theorem 
\ref{thm:ClassifyWithStrongCFP}).

	Like the usual CFP, many C*-algebras have the strong CFP (e.g., see
Propositions \ref{prop:ZStableImpliesUnperf} and \ref{prop:AlmostUnperfImpliesCFP}). 
	Finally, it turns out that the CFP property is closely related to important 
	structural properties of the C*-algebras themselves, like for instance, the Dichotomy question
for simple real rank zero C*-algebras (e.g., see
	\cite{KucerovskyNgSRegular}, \cite{NgCFPSurvey}, \cite{OPRCFP}), and we wonder if this is also so
for the strong CFP.   
	 
\end{para}

\begin{para}\label{df:Extf}
Let $\A$ and $\B$ be C*-algebras with $\B$ nonunital.   We let
$\ee_f(\A, \B)$ denote the set of all unitary equivalence classes
of full extensions $\phi : \A \rightarrow \C(\B)$ with large complement.

If, in addition, $\C(\B)$ is properly infinite, then $\ee_f(\A, \B)$ has
a BDF sum as in \ref{para:BDFSum} for $\ee_{fpi}(\A, \B)$.\\
\end{para}

\begin{thm} \label{thm:ClassifyWithStrongCFP} 
	Let $\A$ be a separable, nuclear C*-algebra, and let $\B$ be a nonunital,
	$\sigma$-unital, simple C*-algebra with the strong CFP and for which $\C(\B)$ is
	properly infinite and $K_1$-injective.
	
	Then the canonical homomorphism
	$$\Gamma : \ee_f(\A, \B) \rightarrow KK(\A, \C(\B)) : [\phi]_{Ext} \mapsto [\phi]_{KK}$$
	is a group isomorphism.   
\end{thm} 

\begin{proof}
	This follows from Theorem \ref{thm:ClassifyFullCase} and the fact that
	since $\B$ has the strong CFP,  every full extension $\A \rightarrow \C(\B)$
	is automatically purely infinite (and hence, $\ee_{fpi}(\A, \B) =
	\ee_f(\A, \B)$).   
\end{proof}

\begin{para} Towards proving that like the usual CFP, many C*-algebras satisfy the strong CFP, we begin by recalling
some notation about the ``nonstabilized" version of the Cuntz semigroup. (See \cite{CuntzW} and \cite{RorZ}.)

Let $\B$ be a C*-algebra. 	
Recall that $\M_{\infty}(\B) =_{df} \bigcup_{n=1}^{\infty} \M_n(\B)$, where for all $n$, the embedding
$\M_n(\B) \hookrightarrow \M_{n+1}(\B) : a \mapsto diag(a, 0)$. 
For $a\in \M_k(\B)_+$ and $b\in \M_l(\B)_+$ set $a\oplus b=diag(a,b)\in \M_{k+l}(\B)_+$, and recall that $a\precsim b$ if there exists a sequence $\{x_m\}_{m=1}^{\infty}$ in $\M_{k+l}(\B)$ such that $x_m b x_m^*\rightarrow a$. Write $a\sim b$ if $a\precsim b$ and $b\precsim a$. Put $W(\B)=_{df} \M_{\infty}(\B)/\sim$ and let $[a]$ denote the $\sim$-equivalence class containing $a$. 

 Define an addition and a partial order on  $W(\B)$ as follows:  For all $a, b \in \M_{\infty}(\B)_+$,
 $[a]+[b] =_{df} [a\oplus b]$ and $[a]\leq [b] \Leftrightarrow a\precsim b$. Then $W(\B)$ is a partially ordered abelian semigroup. 
	
	We say that $W(\B)$ is \emph{almost unperforated} if for all $[a], [b] \in W(\B)$ and all $m,n\in \mathbb{N}$, with $n[a] \leq m[ b] $ and $n>m$, one has that $[a]\leq [b]$.
	
%	For any $a\in \B_{sa}$ and any $\epsilon>0$, let $(a-\epsilon)_+$ denote the positive element defined by the 
%functional calculus of $a$ by the function $f_{\epsilon}(t)=\max\{t-\epsilon, 0\}$.
\end{para}

We begin with a standard computation.

\begin{prop}\label{prop:AlmostUnperfImpliesCFP}
	Let $\B$ be a nonunital $\sigma$-unital C*-algebra for which $\C(\B)$ is properly infinite and $W(\C(\B))$ is almost 
unperforated. 

Then $\B$ has the strong CFP.
\end{prop}

\begin{proof}
	Let $A$ be a full element in $\C(\B)_+$.  Hence, let $\{X_n\}_{n=1}^{k}$ be in $\C(\B)$ such that $1_{\C(\B)}=\sum_{n=1}^{k} X_n A X_n^*$.
	Then $[ 1_{\C(\B)}] \leq k [ A ]$ in $W(\C(\B))$. Since $1_{\C(\B)}$ is properly infinite, $[1_{\C(\B)}]+[ 1_{\C(\B)}] \leq [ 1_{\C(\B)}]$. Then $(k+1)[ 1_{\C(\B)}] \leq k [ A]$ in $W(\C(\B))$.  Since  $W(\C(\B))$ is almost unperforated, it follows
 that $[1_{\C(\B)}] \leq [A]$ in $W(\C(\B))$. Then there exists a sequence $\{ Y_n \}$ in $\C(\B)$ such that 
$Y_n A Y_n^* \rightarrow  1_{\C(\B)}$ in the norm topology.  It follows,  from the continuous functional calculus, that we can 
find $X \in \C(\B)$ for which $XAX^*=1_{\C(\B)}$.  Since $A$ is arbitrary,  $\B$ has the strong CFP.
\end{proof}

The next result follows immediately from \cite{FarahSzabo} Theorem B.  However, since \cite{FarahSzabo}
has not yet been published, we briefly sketch a  proof.

Recall that for a nonunital  C*-algebra $\B$, $\B^{\sim}$ is the minimal unitization of $\D$,
\begin{prop} \label{prop:ZStableImpliesUnperf}
Let $\B$ be a nonunital separable $\Z$-stable C*-algebra.  

Then $\Mul(\B)$ and $\C(\B)$ are separably $\Z$-stable; i.e.,  if $\D \subset \Mul(\B)$ (or $\C(\B)$) is 
a separable C*-subalgebra,
then there is a separable C*-algebra $\D_1$ such that $\D_1$ is $\Z$-stable and $\D \subset \D_1 \subset \Mul(\B)$ ($\C(\B)$ resp.).

As a consequence, $W(\Mul(\B))$ and $W(\C(\B))$ are almost unperforated.

\end{prop}

\begin{proof}[Rough sketch of proof.]
By \cite{Farah} Theorem C, $\C( \B^{\sim} \otimes \Z \otimes \K)$ is separably $\Z$-stable.
Since $\Z$-stability is preserved by extensions (\cite{TomsWinter} Theorem 4.3), $\Mul( \B^{\sim} \otimes \Z \otimes \K)$ is 
separably $\Z$-stable.  By \cite{KucerovskyNgCFPAUE} Proposition 3.1, 
we can find a projection $P \in 
\Mul( \B^{\sim} \otimes \Z \otimes \K)$ such that 
$P( \B^{\sim} \otimes \Z \otimes \K)P \cong \B \otimes \Z \cong \B$.
So $P\Mul( \B^{\sim} \otimes \Z \otimes \K)P \cong \Mul(\B)$.
But since $\Z$-stability is preserved by hereditary C*-subalgebras (\cite{TomsWinter} Corollary 3.1), $P\Mul( \B^{\sim} \otimes \Z \otimes \K)P$ 
and hence, $\Mul(\B)$, are separably $\Z$-stable.  Since $\Z$-stability is preserved by quotients (\cite{TomsWinter} Corollary 3.3), $\C(\B)$
is separably $\Z$-stable.

 We now prove that $W(\C(\B))$ is almost unperforated.  (The proof for $W(\Mul(\B))$ is similar.)
Assume that $[a], [b] \in W(\C(\B))$ and $n > m$ for which $n[a] \leq m [b]$.
Say that $a, b \in \M_L(\C(\B))_+$.   Then there exists a sequence $\{ x_l \}$ in $\M_{nL}(\C(\B))$ for which
$x_l (\bigoplus^m b) x_l^* \rightarrow \bigoplus^n a$. 
Since $\C(\B)$ is separably $\Z$-stable, let $\D_1$ be a separable $\Z$-stable C*-subalgebra of $\C(\B)$ 
such that $a, b \in \M_L(\D_1)$
and $x_l \in \M_{nL}(\D_1)$ fo all $l$.  Hence, $n [a] \leq m[b]$ in $W(\D_1)$.
Since $\D_1$ is $\Z$-stable, by \cite{RorZ} Theorem 4.5, $W(\D_1)$ is almost unperforated. 
Hence, $[a] \leq [b]$ in $W(\D_1)$.
Hence, $[a] \leq [b]$ in $W(\C(\B))$.   Since $[a], [b], m, n$ are arbitrary,   $W(\C(\B))$ is almost unperforated.  Exactly 
the same argument shows that $W(\Mul(\B))$ is almost unperforated.

\end{proof}

\begin{thm}  \label{thm:MainThFullZStable}
	Let $\A$ be a separable nuclear C*-algebra, and $\B$ be a nonunital 
	separable simple $\Z$-stable C*-algebra for which $\C(\B)$ is properly
	infinite.
	
	Then the canonical homomorphism
	$$\Gamma : \ee_f(\A, \B) \rightarrow KK(\A, \C(\B))$$
	is a group isomorphism.  
\end{thm}

\begin{proof}[Sketch of proof]  By Propositions \ref{prop:ZStableImpliesUnperf} and \ref{prop:AlmostUnperfImpliesCFP}, 
$\B$ has the strong CFP.  

Next let us prove that $\C(\B)$ is $K_1$-injective.  Let $u \in \C(\B)$ be a unitary for which $[u]=0$ in $K_1(\C(\B))$.
Since $\C(\B)$ is properly infinite and separably $\Z$-stable (Proposition \ref{prop:ZStableImpliesUnperf}), we can find a  separable $\Z$-stable C*-algebra $\D_1$
such that $\D_1$ is properly infinite, $u \in \D_1 \subset \C(\B)$, $[u] = 0$ in $K_1(\D_1)$, and $1_{\C(\B)} \in \D_1$.  
By \cite{Rohde} Theorem 6.2.3, $\D_1$ is $K_1$-injective, and hence, $u$ is homotopic to $1$ in $U(\D_1)$.
Hence, $u$ is homotopic to $1$ in $U(\C(\B))$.  Since $u$ is arbitrary, $\C(\B)$ is $K_1$-injective.

From the above  and Theorem \ref{thm:ClassifyWithStrongCFP}, the result follows.
	\end{proof}

\begin{para}
 In upcoming papers  (\cite{NgThiel1}, \cite{NgThiel2}), we will prove that if $\B$ is a nonunital 
$\sigma$-unital simple pure C*-algebra for which $\C(\B)$ is properly infinite, then
$W(\C(\B))$ is almost unperforated (which implies that $\B$ has the strong CFP, by 
Proposition \ref{prop:AlmostUnperfImpliesCFP}).  In fact, we will prove that this $\C(\B)$ is pure and 
has real rank zero, which will also
imply that $\C(\B)$ is $K_1$-injective (\cite{LinRR0K1Injective}).  Hence, by Theorem \ref{thm:ClassifyWithStrongCFP},
 in Theorem \ref{thm:MainThFullZStable}, we can replace the hypotheses
of separability and $\Z$-stability of $\B$ with $\sigma$-unitality and purity  
of $\B$

Also, in \cite{NgThiel2}, for a $\sigma$-unital, nonunital, simple, pure C*-algebra $\B$, we provide computations of the K theory 
of $\Mul(\B)$ and $\C(\B)$; and this can be used to give
 concrete examples of the isomorphism in 
Theorem \ref{thm:MainThFullZStable}.
\end{para}

\section{Classifying extensions: The case of purely infinite corona algebras}

In this section, we will classify extensions which, unlike in the previous section (as well
as the classical BDF--Kasparov case), need not be full.  With the present methods, this will necessitate an 
assumption of (not necessarily simple) pure infiniteness for the corona algebra. 
We refer the reader again to Section \ref{sect:SomePreliminaries} 
for some of the preliminary notation and terminologies used here.    

We firstly define the
object that we want to compute.  
The definition itself does not require many assumptions
until
we need to prove more properties and provide classification.

\begin{df} \label{df:ExtSigma} 
	Let $\A$ and $\B$ be separable C*-algebras with $\B$ nonunital.
	Suppose that
	$\sigma : \A \rightarrow \C(\B)$ is an essential extension.
	
	Then $\ee_{\sigma}(\A, \B)$ consists of all $[\phi]$ (see \ref{para:DfUE}) such that 
	\begin{enumerate}
		\item[(a)] $\phi : \A \rightarrow \C(\B)$ is essential and has large complement, and  
		\item[(b)]  $Ideal_{\C(\B)}(\phi(a)) = Ideal_{\C(\B)}(\sigma(a))$ for all
		$a \in \A$. 
	\end{enumerate}  
\end{df}

\begin{para} By an argument similar to that of \ref{para:ExtNonempty},
if $\C(\B)$ is property infinite and $\A$ is separable and
nuclear, then we can always find an essential extension $\sigma : 
\A \rightarrow \C(\B)$;  and for any 
essential extension $\sigma : \A \rightarrow \C(\B)$, $\ee_{\sigma}(\A, \B)$
is nonempty.
\end{para}

\begin{para} \label{para:ExtsigmaSemigroup} 
When $\C(\B)$ is properly infinite, $\ee_{\sigma}(\A, \B)$ has a \emph{BDF sum}
that is (well-) defined  in exactly the same way as in \ref{para:ExtfpiBDFSUM} (see
also \ref{para:BDFSum}).  With this sum, $\ee_{\sigma}(\A, \B)$ is an abelian 
semigroup.  
\end{para}

To get more structure (e.g., group structure) and to prove more results
for $\ee_{\sigma}(\A, \B)$, we
will additionally assume that 
$\C(\B)$ is (not necessarily simple) purely infinite.

\begin{para} \label{para:PIAlgebra}
Recall that a (possibly nonsimple) C*-algebra $\D$ is said to be
\emph{purely infinite} if every nonzero positive element of $\D$ is
properly infinite -- i.e., for all $a \in \D_+$, 
$a \oplus a \preceq a$ in $\M_2 (\D)$ (see   
\ref{para:CuntzSubequiv}).

Simple purely infinite corona algebras were first characterized by Lin
(\cite{LinContScale}), and much subsequent work on extension
theory %(after the BDF--Kasparov theory 
%which centered on absorbing extensions)
focused on this case (e.g., see  
\cite{LinExtII}, \cite{LinExtIII}, \cite{GabeLinNg}).

It then becomes of interest to explore     
the case of nonsimple purely infinite corona algebras.  
We give the following result (Theorem \ref{thm:PICoronaCharacterization}), 
without explaining 
the definitions
or notation,
in order 
to give the reader an indication of the nature of the
theory.  We refer the reader to \cite{KNZPICorona} for definitions,
history and proofs, and  
we also note that the statements, in \cite{KNZPICorona}, of the results
are actually more general.  We specialize to the case of simple pure
canonical ideals to keep matters simple while still being at a relatively
general level, and we note that the statements in the stable canonical ideal
 case
are easier to understand (see the last paragraph in
 Theorem \ref{thm:PICoronaCharacterization}).  

Before going on, we note that for the rest of this paper, whenever we 
talk about a nonunital $\sigma$-unital canonical ideal $\B$, then we assume
that every quasitrace of $\B$ is a trace.
\end{para}

\begin{thm} \label{thm:PICoronaCharacterization}
	Let $\B$ be nonunital $\sigma$-unital simple pure C*-algebra with
	nonzero metrizable tracial simplex $T(\B)$.

	Then the following statements are equivalent:
	\begin{enumerate}
		\item $\B$ has quasicontinuous scale.
		\item $\Mul(\B)$ has strict comparison of positive elements by traces.
		\item $\C(\B)$ is purely infinite.
		\item $\Mul(\B)/\mathcal{I}_{min}$ is purely infinite.
		\item $\Mul(\B)$ has finitely many ideals.
		\item $\C(\B)$ has finitely many ideals.
		\item $\mathcal{I}_{min} = \mathcal{I}_{fin}$.
		\item $V(\Mul(\B))$ has finitely many order ideals. 
	\end{enumerate}

Suppose, in addition, that $\B$ is stable.

Then $\B$ has quasicontinuous scale if and only if $T(\B)$ has finitely 
many extreme points.
\end{thm}

\begin{proof}
	This follows immediately from
	\cite{KNZPICorona} Theorem 7.11 and
Corollary 7.12,  which is a more general result. 
(Again, notice that the stable canonical ideal case is easier to understand.)
\end{proof}

From the above, the hypotheses on the following result, which gives sufficient
conditions for when $\ee_{\sigma}(\A, \B)$ is a group, are reasonable.  

\begin{thm} \label{thm:ExtSigmaIsGroup}
	Let $\A$ be a separable nuclear C*-algebra. Let  $\B$ be a
 separable, nonunital,
	simple, pure C*-algebra with 
	quasicontinuous scale (equivalently, with $\C(\B)$ being purely infinite, by
	Theorem \ref{thm:PICoronaCharacterization}). 
	
	Let $\A \rightarrow \C(\B)$ be an essential extension.
	
	Then $\ee_{\sigma}(\A, \B)$ is an abelian group.\\
	(See \ref{para:ExtsigmaSemigroup}.) 
\end{thm}

\begin{proof}
	
	This follows from \cite{NgExtFunctor2} 
	Theorem 3.2, where the hypotheses (and thus the result) 
are more general. 
\end{proof}

\begin{para}  We note that  with hypotheses as in 
Theorem \ref{thm:ExtSigmaIsGroup} (equivalently, with
hypotheses satisfying any of the conditions in 
Theorem \ref{thm:PICoronaCharacterization}),
there does exist an essential extension $\sigma$ as in the statement
of Theorem \ref{thm:ExtSigmaIsGroup}, and also,
$\ee_{\sigma}(\A, \B)$ is nonempty.
This is proven using the same argument as that of paragraph
\ref{para:ExtNonempty}.
\end{para}

Towards computing the object $\ee_{\sigma}(\A, \B)$ and, thus, to classify
extensions, we will also need to prove some results, which are analogous  
to Propositon \ref{prop:fpiCommutant}, though some technical additions 
are needed:

\begin{lem} \label{lem:ProperlyInfiniteMap}
	Let $\A$ be a separable, unital, nuclear C*-algebra,  and let
	$\B$ be a separable, nonunital, simple, pure C*-algebra with quasicontinuous scale
	(equivalently, with $\C(\B)$ being purely infinite, by Theorem \ref{thm:PICoronaCharacterization}).
	
	Let $\phi: \A \rightarrow \C(\B)$ be a  
	unital,  essential extension. 
	
	Then there exists a partial isometry $v \in \M_2 \otimes \C(\B)$ such that 
	$$v (\phi(a) \oplus 0) v^*  = \phi(a) \oplus \phi(a) 
\makebox{ in } \M_2 \otimes \C(\B)$$
	for all $a \in \A$.  (Here, $\oplus$ is the usual 
``diagonal sum".)
\end{lem}

\begin{proof}
	Since $\C(\B)$ is properly infinite, let $S, T \in \C(\B)$ be isometries such that
	$SS^* + TT^* \leq 1_{\C(\B)}$.  (Necessarily, $S$ and $T$ have orthogonal ranges.)
	Now consider the *-homomorphism $\psi: \A \rightarrow \C(\B)$ that is given by
	$\psi(a) =_{df} S \phi(a) S^* + T \phi(a) T^* \makebox{ for all  } a \in \A.$
	(I.e., $\psi$ is the BDF sum of $\phi$ with itself. See \ref{para:BDFSum}.)
	Note that $\psi(a) \in Ideal(\phi(a))$ for all $a \in \A$.
	
	By \cite{NgExtFunctor2} Theorem 2.12, let $w \in \C(\B)$ be such that 
	$w \phi(a) w^* = \psi(a) \makebox{ for all } a \in \A.$ 
	Since $\phi$ is unital, $w w^* = S S^* + TT^*$ is a projection 
	(in $\C(\B)$), and
	hence, $w$ is a partial isometry.
	
	To simplify notation, 
	let $\widetilde{S}, \widetilde{T}, \widetilde{w}
	\in \M_2 \otimes \C(\B)$ be given by  
	$\widetilde{S} =_{df} S \oplus 0, \makebox{ }
	\widetilde{T} =_{df} T \oplus 0, \makebox{ and }
	\widetilde{w} =_{df} w \oplus 0$ in $\M_2 \otimes \C(\B)$, 
where $\oplus$ is the
usual ``diagonal sum". 
	Noting that $\widetilde{e}_{2,1} =_{df} \left[\begin{array}{cc} 0 & 0 \\
		1 & 0\end{array}\right] 
	\in
	\M_2 \otimes \C(\B)$ is a partial 
	isometry that witnesses that $1 \oplus 0 
	\sim 0 \oplus 1$ in $\M_2 \otimes \C(\B)$, let 
	$v \in \M_2 \otimes \C(\B)$ be the partial isometry that is defined by
	$v =_{df} (\widetilde{S}^* + \widetilde{e}_{2,1} \widetilde{T}^*)
	\widetilde{w}.$
	Then for all $a \in \A$,
	we have that
	$$v (\phi(a) \oplus 0) v^* = \phi(a) \oplus \phi(a) \makebox{ in }
	\M_2 \otimes \C(\B),$$ 
where $\oplus$ is the ``diagonal sum".
\end{proof}

\begin{prop} \label{prop:PICoronaCommutant}
	Let $\A$ be a unital
 separable nuclear C*-algebra,  and let
	$\B$ be a separable, nonunital, simple, pure C*-algebra with quasicontinuous scale
	(equivalently, with $\C(\B)$ being purely infinite, by Theorem \ref{thm:PICoronaCharacterization}).
	
	Let $\phi : \A \rightarrow \C(\B)$ be a unital essential extension.   
	
	Then $\phi(\A)'$ is properly infinite.
\end{prop} 

\begin{proof}
	The proof is exactly the same as that of  
	Proposition \ref{prop:fpiCommutant}, except that we 
	replace (\ref{equ:June29202010AM}) -- i.e., the application of
	\cite{NgExtFunctor2} Theorem 2.9 -- by this paper's Lemma
	\ref{lem:ProperlyInfiniteMap}.  
	\end{proof}

Continuing to move  
towards computing the object $\ee_{\sigma}(\A, \B)$,  
we will briefly
recall some notations and notions for actions of a topological space
on a C*-algebra and Kirchberg's interesting 
ideal related KK theory (See \cite{KirchbergIdealKK};  see also 
\cite{GabeOinfty}, \cite{MeyerNest1}). We only (quickly) 
introduce the minimally
necessary notions,
referring to \cite{GabeOinfty} as our main reference, and for more 
details.

\begin{para} \label{para:O(X)I(D)} 
Recall that a \emph{complete lattice} is a partially ordered set where every
subset has both a supremum and an infimum (e.g., see \cite{GabeOinfty}   
subsection 10.1).  For a topological space $X$, let $\oo(X)$ 
denote the collection 
of open subsets of $X$. Then $\oo(X)$ is partially ordered, 
where the order relation is the subset 
relation. In fact,  $\oo(X)$ 
is a complete lattice where infimum and supremum are   
defined in the following ways:  Let $\{ G_{\alpha} \}_{\alpha \in I}$ be
a collection in $\oo(X)$ (i.e., $\{ G_{\alpha} \}_{\alpha \in I}$
is a collection of open subsets of $X$). 
Then the infimum of $\{ G_{\alpha} \}_{\alpha \in I}$ is 
$Int(\bigcap_{\alpha \in I} G_{\alpha})$ (i.e., the interior of 
$\bigcap_{\alpha \in I} G_{\alpha}$). (See \cite{GabeOinfty} footnote 38.)
Also, the supremum of $\{ G_{\alpha} \}_{\alpha \in I}$ is obviously
$\bigcup_{\alpha \in I} G_{\alpha}$. 

Similarly, for a C*-algebra $\D$,
if we let $\I(\D)$ denote the partially ordered
collection of closed two-sided ideals of $\D$ (again
where the order relation is the subset relation), then $\I(\D)$ is a complete
lattice where the supremum and infimum are defined in the following ways:  Let 
$\{ \J_{\alpha} \}_{\alpha \in I}$ be a collection in $\I(\D)$ (i.e., 
$\{ \J_{\alpha} \}_{\alpha \in I}$ is a 
collection of closed two-sided ideals of $\D$).  Then the supremum of 
$\{ \J_{\alpha} \}_{\alpha \in I}$ is 
$\overline{\sum_{\alpha \in I} \J_{\alpha}}$  (the norm closure of the 
algebraic sum (finite sums) of all the $\J_{\alpha}$).   
(See \cite{GabeOinfty} equation (10.2).)   
Also, obviously, the infimum of $\{ \J_{\alpha} \}_{\alpha \in I}$ is
$\bigcap_{\alpha \in I} \J_{\alpha}$. 
\end{para}

\begin{para} \label{para:Actions} Next, we recall  
some basic notations and ideas concerning actions of topological spaces on 
C*-algebras.  Our main reference for this material
is \cite{GabeOinfty} subsection 10.2 (supplemented with subsections
10.1 and 10.3).  
Let 
$X$ be a topological space, and let 
$\D$ be a C*-algebra. Recall that  $\oo(X)$ is the  
(complete) lattice of open subsets of $X$, and  $\I(\D)$ is  the
(complete) lattice of closed two sided ideals of $\D$ (see 
the previous paragraph \ref{para:O(X)I(D)}).  
An \emph{action} of $X$ on $\D$ is an order preserving map
$\Phi_{\D} : \oo(X) \rightarrow \I(\D)$.  
The pair $(\D, \Phi_{\D})$ is called an \emph{$X$-C*-algebra} -- though
sometimes, when the context (action) is clear,
we omit the $\Phi_{\D}$ and just call $\D$ an $X$-C*-algebra.   
To save notation, we also sometimes write 
$\D(U) =_{df} \Phi_{\D}(U)$ for every $U \in \oo(X)$.   
Suppose that $\E$ is another C*-algebra and $\Phi_{\E} : \oo(X) 
\rightarrow \I(\E)$ is an action of $X$ on $\E$ (i.e.,    
$(\E, \Phi_{E})$ is another $X$-C*-algebra). A map   
$\phi : \D \rightarrow \E$ is called \emph{$X$-equivariant} 
if $$\phi(\A(U)) \subseteq \E(U) \makebox{ for all } U \in \oo(X)$$ 
or, equivalently, $\phi \circ \Phi_{\D} \subseteq \Phi_{\E}$ as set maps.
\end{para}

\begin{para} \label{para:FullXMap}
Recall the definition of a full equivariant map as in 
\cite{GabeOinfty}, the second paragraph after equation (1.8) (see
also \cite{GabeOinfty} Remark 10.19):   
Suppose that $X$, $\D$, $\E$, $\Phi_{\D}$ and
$\Phi_{\E}$ are as in the previous paragraph \ref{para:Actions} (i.e., 
we are assuming that $(\D, \Phi_{\D})$ and $(\E, \Phi_{\D})$ are 
$X$-C*-algebras). Suppose, in addition, that for every $d \in \D$, there
exists a smallest $U_d \in \oo(X)$
for which $d \in \Phi_{\D}(U_d)$ 
(i.e, there exists a smallest open subset $U_d \subseteq X$
for which $d \in \D(U_d)$).  Then an $X$-equivariant *-homomorphism 
$\phi : \D \rightarrow \E$ is said to be \emph{$X$-full} if for all
$d \in \D$, $\phi(d)$ is full in $\E(U_d) = \Phi_{\E}(U_d)$.
(Note that the assumption that $U_d$ exists, for all $d \in \D$, 
is true when $\Phi_{\D}$ is \emph{lower semicontinuous}.  See 
\cite{GabeOinfty} Definition 10.8 and the paragraph after it.)
\end{para}

\begin{lem} \label{lem:Oct3020242AM}
	Let $\D$ and $\E$ be C*-algebras, and 
	let $\phi : \D \rightarrow \E$ be a  *-homomorphism.
	Let $X$ be the topological space given by $X =_{df} Prim(\E)$, i.e., 
	$X$ is the primitive ideal space of $\E$, endowed with the Jacobson or
	hull--kernel topology  
	(see \cite{DavidsonBook}
	section VII.3 or \cite{PedersenBook} subsection 4.1.2).   
	
	Let $\phi : \D \rightarrow \E$ be a  *-homomorphism,   
	 and
	let $\Phi_{\D} : \oo(X) \rightarrow \I(\D)$ and $\Phi_{\E} : \oo(X) \rightarrow
	\I(\E)$ be the 
actions of $X$, on $\D$ and $\E$ respectively, that are given as
	follows:
	
	\begin{equation} \label{equ:Oct3020241AM} 
		\Phi_{\E} : G \mapsto \bigcap_{\J \in X - G} \J \makebox{ and }
		\Phi_{\D} : G \mapsto \phi^{-1}(\Phi_{\E}(G)).
	\end{equation}
	
	Then the following statements are true:
	\begin{enumerate} 
		\item $\Phi_{\E}$ is an  order isomorphism (and hence, is bijective). 
		\item $\Phi_{\D}$ preserves  infima 
		 (i.e., $\Phi_{\D}$
preserves infima of arbitrary subsets of
		$\oo(X)$).  If, additionally, $\oo(X)$ is finite, then
$\Phi_{\D}$ preserves suprema of nonempty increasing nets 
		in $\oo(X)$.    
		\item For every $d \in \D$, there exists a smallest  open subset
		$U_d \subseteq X$ for which
		$d \in \D(U_d)$.   
		\item The *-homomorphism $\phi$ is $X$-equivariant and $X$-full.
		\item Suppose that $\psi : \D \rightarrow \E$ is a *-homomorphism. 
		Then 
		\begin{equation}  \label{equ:Oct1820243AM}
			Ideal(\psi(a)) = Ideal(\phi(a)) \makebox{ for all }a \in \D  
		\end{equation} 
		if and only if 
		$\psi$ is $X$-equivariant 
and $X$-full. 
\end{enumerate}

\end{lem}

\begin{proof}
	
	Statement (1) follows from \cite{PedersenBook} Theorem 4.1.3 and its proof.
	
	For statement (2), let us firstly prove that $\Phi_{\D}$ preserves
	infima.  Let $\{ G_{\alpha} \}_{\alpha \in I}$ be a collection in 
	$\oo(X)$.   
	We have that 
	\begin{eqnarray*}
		\Phi_{\D}\left( \inf\{ G_{\alpha} : \alpha \in I \}
\right)& = & 
		\phi^{-1}\left(\Phi_{\E}\left(\inf\{ G_{\alpha}
	: \alpha \in I\}\right)\right) 
\makebox{  (definition of }\Phi_{\D}\makebox{)}\\ 
	& = & \phi^{-1}\left(\inf\{\Phi_{\E}(G_{\alpha}) : \alpha \in I\}\right)
\makebox{ (by (1), }\Phi_{\E}\makebox{ preserves infima.)}\\    
& = & \phi^{-1}\left(\bigcap_{\alpha \in I} \Phi_{\E}(G_{\alpha})\right)
		\makebox{  (definition of infimum in $I(\E)$; see
\ref{para:O(X)I(D)}.)}\\
		& = & \bigcap_{\alpha \in I} \phi^{-1}(\Phi_{\E}(G_{\alpha}))\\ 
		& = & \bigcap_{\alpha \in I} \Phi_{\D}(G_{\alpha}) 
		\makebox{ (definition of  }\Phi_{\D}\makebox{)}\\
& = & \inf\{ \Phi_{\D}(G_{\alpha}) : \alpha \in I\}
\makebox{  (definition of infimum in $I(\D)$; see \ref{para:O(X)I(D)})}. 
	\end{eqnarray*}
	Hence, since $\{ G_{\alpha} \}_{\alpha \in I}$ is arbitrary, $\Phi_{\D}$
	preserves infima.

      For the second statement in (2), assume that $\oo(X)$ is finite.
Suppose that $\{ G_{\alpha} \}_{\alpha \in I}$ is an increasing net in 
$O(X)$ (so $I$ is a directed set; and for all $\alpha, \alpha' \in I$,
if $\alpha \leq \alpha'$ then $G_{\alpha} \subseteq G_{\alpha'}$).  
Since $\oo(X)$ is finite, we can choose $\beta \in I$ such that for all
$\alpha \in I$ with $\alpha \geq \beta$, $G_{\alpha} = G_{\beta}$.
Hence, 
$$
\Phi_{\D}(\sup\{ G_{\alpha} : \alpha \in I\})
 =  \Phi_{\D} (G_{\beta})
 =  \sup\{ \Phi_{\D}(G_{\alpha}) : \alpha \in I\}.   
$$ 
Since $\{ G_{\alpha} \}_{\alpha \in I}$ is arbitrary, if $\oo(X)$
is finite then $\Phi_{\D}$ preserves increasing suprema. 
	
	We now prove (3).   Let $d \in \D$.  Consider
	the collection $\mathcal{S} =_{df} \{ G \in \oo(X) : d \in \Phi_{\D}(G) \}$.
	By the definition of $\Phi_{\D}$, 
	\begin{equation} \label{equ:Oct1820241AM}
		\mathcal{S} = \{ G \in \oo(X) : \phi(d) \in \Phi_{\E}(G) \}.  
	\end{equation} 
	Hence, by (1),  $\mathcal{S} \neq \emptyset$.  
	By (2), $\Phi_{\D}$ preserves arbitrary infima.  Hence, if we let  
	$U_d \in \oo(X)$  be given by 
	$U_d =_{df} \inf (\mathcal{S}) =_{df} Int \left(\bigcap \mathcal{S}
\right)$, then 
	\begin{equation} \label{equ:Oct1820242AM} 
		\Phi_{\D}(U_d) = \inf \Phi_{\D}(\mathcal{S}) =_{df}  
\bigcap_{G \in \mathcal{S}} \Phi_{\D}(G)
	\end{equation}  
	and the latter set contains $d$ by the definition of $\mathcal{S}$.  
	Hence, $d \in \Phi_{\D}(U_d)$.  
	Hence, by the definitions of $\mathcal{S}$ and $U_d$, 
	$U_d$ is the smallest open subset of $X$ for which $d \in \Phi_{\D}(U_d)$.

	Next, let us now prove (4).  
	Firstly, for all $G \in \oo(X)$, by the definition 
	of $\Phi_{\D}$, we have that
	$$\phi(\Phi_{\D}(G)) = \phi(\phi^{-1}(\Phi_{\E}(G))) \subseteq
	\Phi_{\E}(G),$$
	and hence, $\phi$ is $X$-equivariant. 
	
	Also, let $d \in \D$ be arbitrary.  Let $U_d \in \oo(X)$ and 
	$\mathcal{S}$ be as in (3) and its proof. 
	By (1) and (\ref{equ:Oct1820241AM}), we have that 
	\begin{equation} \label{equ:Champagnes}
Ideal(\phi(d)) = \bigcap_{G \in  \mathcal{S}} \Phi_{\E}(G).
\end{equation}   
But by (1), $\Phi_{\E}$ preserves infima.  Hence,
$\Phi_{\E}(\inf(\mathcal{S})) = \inf \Phi_{\E}(\mathcal{S})$.  
Hence, by the definition of $U_d$ (just before  
(\ref{equ:Oct1820242AM})) and by \ref{para:O(X)I(D)},
$\Phi_{\E}(U_d) = \bigcap_{G \in \mathcal{S}} \Phi_{\E}(G)$.
From this and (\ref{equ:Champagnes}), $Ideal(\phi(d)) = \Phi_{\E}(U_d)$. 
I.e., 
	$\phi(d)$ is full in $\Phi_{\E}(U_d)$.  Since $d \in \D$ is arbitrary,
	$\phi$ is $X$-full.

	Let us now prove the ``only if" direction of (5). Suppose that
$\psi : \D \rightarrow \E$ is a *-homomorphism such that $Ideal(\phi(a)) =
Ideal(\psi(a))$ for all $a \in \D$, i.e., suppose that
(\ref{equ:Oct1820243AM}) holds.    Let $G \in \oo(X)$.
	Since $\phi$ is $X$-equivariant, 
$\phi(\Phi_{\D}(G)) \subseteq \Phi_{\E}(G)$.
	So $Ideal(\phi(\Phi_{\D}(G))) \subseteq \Phi_{\E}(G)$.
	But by our hypothesis (\ref{equ:Oct1820243AM}),
$	
Ideal(\phi(\Phi_{\D}(G))) = Ideal(\psi(\Phi_{\D}(G))).
$
	Hence, $\psi(\Phi_{\D}(G)) \subseteq \Phi_{\E}(G)$.  
Since $G \in \oo(X)$
	is arbitrary, $\psi$ is $X$-equivariant.
	
	Let $d \in \D$ be arbitrary.   By (3), 
let $U_d \in \oo(X)$ be the smallest
	open subset of $X$ for which $d \in \Phi_{\D}(U_d)$.
	Since $\phi$ is $X$-full, 
	$\phi(d)$ is full in $\Phi_{\E}(U_d)$.  But by our
hypothesis (\ref{equ:Oct1820243AM}),
	$Ideal(\psi(d)) = Ideal(\phi(d)) = \Phi_{\E}(U_d)$.  Hence, $\psi(d)$
	is full in $\Phi_{\E}(U_d)$.  Since $d$ was arbitrary, $\psi$ is 
	$X$-full.  
	
	Finally, let us prove the ``if" direction of (5).  Suppose that 
	$\psi : \D \rightarrow \E$ is an $X$-equivariant and $X$-full
*-homomorphism.
	Let $d \in \D$ be arbitrary.
	By (3), let $U_d \in \oo(X)$ be the smallest open subset of $X$ 
	for which $d \in \Phi_{\D}(U_d)$.  Since $\psi$ is $X$-full,
	$\psi(d)$ is full in $\Phi_{\E}(U_d)$ (see  
	\ref{para:FullXMap}). Hence, $Ideal(\psi(d)) = \Phi_{\E}(U_d)$.
	But by (4), $\phi$ is also $X$-equivariant and $X$-full. Hence,
	$Ideal(\phi(d)) = \Phi_{\E}(U_d) = Ideal(\psi(d))$.   
	Since $d \in \D$ was arbitrary, we are done.  
\end{proof}

\begin{para} \label{para:DfKKX} We will need 
Kirchberg's interesting ideal related KK theory (\cite{KirchbergIdealKK}).  
We 
briefly introduce the subject, giving the ideal analogue of
the generalized homomorphism (or Cuntz) picture of 
KK. For more details, 
we refer the reader to \cite{GabeOinfty} 
Section 12.3 (though our
terminology will slightly differ from theirs;  see
also \cite{MeyerNest1}).  Roughly speaking,
ideal related KK theory is the same as KK theory, except that we require
that all maps be equivariant in the appropriate sense.

Let us firstly quickly recall 
the generalized homomorphism (or Cuntz) picture 
of KK (see \cite{JensenThomsen} Chapter 4).  
Let $\D$ and $\E$ be C*-algebras with $\D$ separable and $\E$ $\sigma$-unital. 
Recall that a  \emph{$KK_h(\D, \E)$-cocycle} is a pair $(\phi, \psi)$ 
of *-homomorphisms $\D \rightarrow \Mul(\E \otimes \K)$ for which
$\phi(d) - \psi(d) \in \E \otimes \K$ for all $d \in \D$.  Recall also,
that two $KK_h(\D, \E)$-cocycles $(\phi_0, \psi_0)$ and 
$(\phi_1, \psi_1)$ are said to be \emph{homotopic} if  
there exists a path $\{ (\phi_t, \psi_t)\}_{t \in [0,1]}$ such that
the following statements are true:  
\begin{itemize}
	\item[i.] $(\phi_t, \psi_t)$ is a $KK_h(\D, \E)$-cocycle for all $t \in [0,1]$.
	\item[ii.] For all $d \in \D$,
	the two maps $[0,1] \rightarrow \Mul(\E \otimes \K)$ given by 
	$t \mapsto \phi_t(d)$ and $t \mapsto \psi_t(d)$ are strictly continuous.    
	\item[iii.] For all $d \in \D$, the map
	$[0,1] \rightarrow \E \otimes \K: t \mapsto \phi_t(d) - \psi_t(d)$
	is norm continuous.                            
\end{itemize}
We let 
$KK(\D, \E)$ denote the homotopy classes of $KK_h(\D, \E)$-cocycles.
$KK(\D, \E)$ has an obvious addition which makes it an abelian group,
and any *-homomorphism $\xi : \D \rightarrow \E$
can induce  a KK-element
by defining  $[\xi]_{KK} =_{df} [(\xi(\cdot) \otimes e_{1,1}, 0)] \in
KK(\D, \E)$. (Here, $\xi(\cdot) \otimes e_{1,1} : \D \rightarrow \E \otimes 
\K \subset  \Mul(\E \otimes \K)$.)   
See \cite{JensenThomsen} Chapter 4 for more details.

To move from KK to Kirchberg's ideal related KK, 
we now bring in an action of a topological space (see 
paragraphs \ref{para:O(X)I(D)}, \ref{para:Actions}). 
We will follow   
\cite{GabeOinfty} Subsection 12.3 and refer to it for more details, 
though we use slightly
different terminology.  
Now let $X$ be a topological space, and suppose, additionally,
that $\D$ and $\E$ are $X$-C*-algebras (see \ref{para:Actions}).      
Then $\E \otimes \K$ is an $X$-C*-algebra, with $X$-action induced
from that for $\E$ in the obvious way. 
A map $\phi : \D \rightarrow \Mul(\E \otimes \K)$ is said to be \emph{weakly 
	$X$-equivariant} if for all $x \in \E \otimes \K$, the map 
$\D \rightarrow \E \otimes \K : d \rightarrow x \phi(d) x^*$ is $X$-equivariant (see 
\ref{para:Actions};  also see \cite{GabeOinfty} Definition 11.1,
Remark 11.2 and Corollary 11.12).  
A $KK_h(\D, \E)$-cocycle $(\phi, \psi)$ is said to be 
\emph{weakly $X$-equivariant} if  $\phi$ and $\psi$ are both weakly  
$X$-equivariant.  Two weakly $X$-equivariant $KK_h(\D, \E)$-cocycles   
$(\phi_0, \psi_0)$ and $(\phi_1, \psi_1)$ are said to be
\emph{$X$-homotopic} if they are homotopic (as in the previous
paragraph; see items i., ii., and iii.) via a path $\{ (\phi_t, \psi_t) \}_{t \in [0,1]}$ such 
that $(\phi_t, \psi_t)$ is weakly $X$-equivariant for all
$t \in [0,1]$. 
$KK(X; \D, \E)$ consists of all $X$-homotopy
classes of weakly $X$-equivariant $KK_h(\D, \E)$-cocycles.
Again, with an obvious addition, $KK(X; \D, \E)$ is an abelian
group, and a given $X$-equivariant *-homomorphism
$\xi : \D \rightarrow \E$ induces an element   
$[\xi]_{KKX} =_{df} [(\xi(\cdot) \otimes e_{1,1}, 0)]
\in KK(X; \D, \E)$.  
(It is not hard to see that $\xi(\cdot) \otimes e_{1,1} : \D \rightarrow
\E \otimes \K \subset \Mul(\E \otimes \K)$ is weakly $X$-equivariant.)   
See \cite{GabeOinfty} Section 12.3 for
more details.
\end{para}

\begin{thm} \label{thm:MainThEnd}
	Let $\A$ be a separable nuclear C*-algebra, and let $\B$ be a 
	nonunital, separable, simple, pure C*-algebra with quasicontinuous
	scale (equivalently, with $\C(\B)$ being purely infinite by
	Theorem \ref{thm:PICoronaCharacterization}).  
	
	Let $X =_{df} Prim(\C(\B))$ (i.e., $X$ is the primitive
ideal space of $\C(\B)$,  given
the Jacobson topology), and let $\sigma : \A \rightarrow
	\C(\B)$ be an injective *-homomorphism.  
Let 
$\Phi_{\A} : O(X) \rightarrow \A$ and
	$\Phi_{\C(\B)} : O(X) \rightarrow \C(\B)$ be the actions on 
	$\A$ and $\C(\B)$ respectively, which are defined as in 
	(\ref{equ:Oct3020241AM}) (in the hypothesis of Lemma
	\ref{lem:Oct3020242AM}, with $\D$, $\E$, $\phi$ replaced with
	$\A$, $\C(\B)$, $\sigma$ respectively).
	
	Then the canonical homomorphism
	$$\Gamma : \ee_{\sigma}(\A, \B) \rightarrow KK(X; \A, \C(\B)) :
	[\psi]_{Ext} \mapsto [\psi]_{KKX}$$
	is a group isomorphism.   
	
	(See Theorem \ref{thm:ExtSigmaIsGroup} and 
	\ref{para:DfKKX}.)
\end{thm}

\begin{proof}
	Note that by the definition of $\ee_{\sigma}(\A, \B)$ (see 
	Definition \ref{df:ExtSigma}) and by Lemma \ref{lem:Oct3020242AM}
	item (5),  
	if $[\psi]_{Ext} \in 
	\ee_{\sigma}(\A, \B)$, then $\psi: \A \rightarrow \C(\B)$ is 
	$X$-equivariant, and hence, $\psi$ induces an element
	$[\psi]_{KKX} \in KK(X; \A, \C(\B))$. (See the end of the last paragraph
	of \ref{para:DfKKX}.)  
	Hence, that $\Gamma$ is a well-defined group homomorphism follows from the
	above and from 
	the basic properties of ideal related KK theory (see, for example,
	\ref{para:DfKKX} and \cite{GabeOinfty} Section 12).

	We next prove the surjectivity of the map $\Gamma$.
	Our first goal is to  check that the hypotheses of the existence theorem 
	of \cite{GabeOinfty} Proposition 15.6 are satisfied.   
	Firstly,  $\A$ is a separable nuclear C*-algebra. Also,  by 
	Lemma \ref{lem:Oct3020242AM} item (2),  
	 $\A$ is a lower semicontinuous $X$-C*-algebra in the sense of 
	\cite{GabeOinfty} Definition 10.8.

	Next, since $\C(\B)$ is properly infinite, it contains a unital copy of
	$O_{\infty}$, and hence, we have a sequence $\{ p_n \}_{n=1}^{\infty}$
	of pairwise orthogonal projections in $\C(\B)$ such that for all $n$,
	$p_n \sim 1_{\C(\B)}$.  Let $d \in \C(\B)$ be given by 
	$d =_{df} \sum_{n=1}^{\infty} \frac{p_n}{n}$.  Then 
	$\D_1 =_{df} \overline{d \C(\B) d}$ is a $\sigma$-unital full hereditary
	C*-subalgebra of $\C(\B)$.  Moreover,  by the definition of $d$
and since $d$ is a strictly positive element of $\D_1$,
 it is not hard to check that 
	$\D_1$ satisfies the 
	characterization of stability in \cite{HjelmborgRordam} Theorem 3.3.
Hence, by 
	\cite{HjelmborgRordam} Theorem 3.3, $\D_1$ is stable.  Hence, $\D_1$ is a $\sigma$-unital,
	full, stable, hereditary
	C*-subalgebra of $\C(\B)$.

	Next, by Lemma \ref{lem:Oct3020242AM} item (4), 
$\sigma : \A \rightarrow
	\C(\B)$ is $X$-equivariant and  $X$-full.
Finally, recall that $\A^+$ is the forced unitization of $\A$ (see 
\ref{para:ForcedUnitization}). 
If we let $\sigma^+ : \A^+ \rightarrow \C(\B)$ be the unique unital
extension of
$\sigma$, then by Proposition \ref{prop:PICoronaCommutant}, 
$\sigma^+$ is strongly
	$O_{\infty}$-stable in
	the sense of \cite{GabeOinfty} (1.1).  
Hence, $\sigma$ is also strongly
        $O_{\infty}$-stable in
        the sense of \cite{GabeOinfty} (1.1).

	From the above, the hypotheses of \cite{GabeOinfty} Proposition 15.6 are
	satisfied.  Now let $\alpha \in KK(X; \A, \C(\B))$ be arbitrary.
	By \cite{GabeOinfty} Proposition 15.6,  
	there exists an $X$-equivariant and $X$-full *-homomorphism
	$\phi : \A \rightarrow \C(\B)$ such that 
	$[\phi]_{KKX} = \alpha$. 
	Replacing $\phi$ with $\phi \oplus 0$ 
(where $\oplus$ is the BDF sum; see \ref{para:ExtsigmaSemigroup}) 
 if 
	necessary, we may assume that $\phi$ has large complement.
	Moreover, by Lemma \ref{lem:Oct3020242AM} item (5),   
	$Ideal(\phi(a)) = Ideal(\sigma(a))$ for all $a \in \A$ (which implies
	that $\phi$ is injective, since $\sigma$ is injective -- i.e., 
	$\phi$ is an essential extension).
	Thus, $\phi$ induces an element
	$[\phi]_{Ext} \in \ee_{\sigma}(\A, \B)$ (see Definition
	\ref{df:ExtSigma}) and $\Gamma([\phi]_{Ext})  = \alpha$.
	Since $\alpha \in KK(X; \A, \B)$ is arbitrary, $\Gamma$ is surjective.

	We now prove that $\Gamma$ is injective.  
	Again, let us firstly the check that the hypotheses of the uniqueness
	theorem \cite{GabeOinfty} Theorem F are satisfied. 
	Firstly, note that	
	$\A$ is a separable nuclear C*-algebra.  Also, by Lemma
	\ref{lem:Oct3020242AM} item (2), $\A$ is a lower semicontinuous
	$X$-C*-algebra in the sense of \cite{GabeOinfty} 10.8.

	Now let $[\phi]_{Ext}, [\psi]_{Ext} \in \ee_{\sigma}(\A, \B)$ 
be such that
	$\Gamma([\phi]_{Ext}) = \Gamma([\psi]_{Ext})$; i.e., 
	$[\phi]_{KKX} = [\psi]_{KKX}$ in $KK(X; \A, \C(\B))$.
	So $\phi$ and $\psi$ both are essential extensions
with large complement and 
	$Ideal(\phi(a)) = Ideal(\sigma(a)) = Ideal(\psi(a))$ 
for all $a \in \A$.   
Hence, by Lemma \ref{lem:Oct3020242AM} item (5), $\phi$ and $\psi$ are both $X$-equivariant and $X$-full.  
      
Also, by taking the unique unital extensions $\phi^+$ and $\psi^+$, of
$\phi$ and $\psi$ (resp.), to $\A^+$, and by applying Proposition 
\ref{prop:PICoronaCommutant}, we see that $\phi$ and $\psi$ 
are strongly $O_{\infty}$-stable in the sense of 
\cite{GabeOinfty} (1.1).

Hence, all the hypotheses of \cite{GabeOinfty} Theorem F are fullfilled.
Hence, by \cite{GabeOinfty} Theorem F, $\phi$ and $\psi$ are asymptotically
Murray--von Neumann equivalent in the sense of \cite{GabeOinfty} (1.6).
Hence, since $\C(\B)$ is properly infinite and $\phi$ and $\psi$ both have large complement, 
and by \cite{GabeO2}   
Proposition 3.10, $\phi$ and $\psi$ are asymptotically unitarily equivalent.
Hence, by Lemma \ref{lem:PhillipsWeaver1},
$\phi$ and $\psi$ are weakly unitarily equivalent.  Note
that by \cite{ChandNgSutradhar} Theorem 4.9 (4), $\C(\B)$ is $K_1$-injective.
Hence, by \cite{GabeLinNg} Proposition 3.7 (see also
\cite{ChandNgSutradhar} Lemma 2.2 for another proof), it follows that
$\phi$ and $\psi$ are unitarily equivalent.  Hence, $[\phi]_{Ext} =
[\psi]_{Ext}$.  Since $[\phi]_{Ext}$ and $[\psi]_{Ext}$ 
are arbitrary,    $\Gamma$
is injective.

\end{proof}

\begin{para}
After Theorem \ref{thm:MainThEnd}, 
a next task is to provide computations for ideal related KK-theory.
For the original Kasparov KK theory, this often means providing a 
Universal Coefficient Theorem (\cite{RosenbergSchochet}). 
But there is no general Universal Coefficient Theorem for ideal related
KK theory (e.g., see \cite{Bonkat}, \cite{Bentmann}, \cite{EilersRestorffRuiz},
\cite{MeyerNest2}, \cite{LinExtIII}).  
Moreover, we do not yet have a complete picture for the K theory
of the ideals, and other features.  
Hence, 
we leave these considerations for the future.
\end{para}
\ \\

{\bf Acknowledgements}\\

The second author was supported by the National Natural Science Foundation of China (Grant No. 12201216) and the Visiting Scholars Program of the Chern Institute of Mathematics at Nankai University. The generous support from Prof. Chi-Keung Ng and the supportive research environment provided by the institute are gratefully acknowledged.

\end{document}